\documentclass[a4paper,11pt]{amsart}

\usepackage{amsfonts}
\usepackage{amsmath}
\usepackage{amssymb}
\usepackage{amsthm}
\usepackage[francais]{babel}
\usepackage[T1]{fontenc}
\usepackage{mathrsfs}
\usepackage[dvips]{hyperref}
\usepackage{enumerate}
\usepackage{amssymb}
\usepackage{amsmath}
\usepackage{pstricks}

\usepackage{graphicx}
\newtheorem{lemme}{Lemme}

\newtheorem{theo}{Th\'eor\`eme}
\newtheorem{proposition}{Proposition}
\newtheorem{corollaire}{Corollaire}

\theoremstyle{definition}
\newtheorem{remarque}{Remarque}
\newtheorem{notation}{Notation}

\newtheorem{defi}{D\'efinition}

\usepackage{graphicx}

\def\F{\mathbb{F}}
\def\N{\mathbb{N}}

\begin{document}
%\maketitle
\title[] 
{Sur le d\'eveloppement en fraction continue d'une g\'en\'eralisation de la cubique de Baum et Sweet }

\author{Alina Firicel}

\address{Universit\'{e} de Lyon\\
Universit\'{e} Lyon 1 \\
 Institut Camille Jordan\\
  UMR 5208 du CNRS \\
   43, boulevard du 11 novembre 1918 \\ 
   F-69622 Villeurbanne Cedex, France\\
\texttt{firicel@math.univ-lyon1.fr}}

\maketitle

\section{Introduction}
\par Il y a une trentaine d'ann\'ees, les travaux de Baum et Sweet
\cite{baumsweet} ont ouvert un nouveau domaine de recherche sur l'approximation
diophantienne dans les corps de s\'eries formelles \`a coefficients dans un corps fini,
par le biais du d\'eveloppement en fraction continue. Ces auteurs ont
notamment donn\'e l'exemple d'une s\'erie formelle \`a coefficients dans le corps fini $\F_2$,
alg\'ebrique de degr\'e $3$ sur $\F_2(T)$, ayant un d\'eveloppement en fraction continue
avec des quotients partiels qui sont tous des polyn\^{o}mes en $T$ de degr\'e $1$ ou $2$. Dix ans
plus tard, Mills et Robbins \cite{millsrobbins} ont d\'ecrit un algorithme qui leur a
permis de donner le d\'eveloppement explicite en fraction continue pour
la s\'erie formelle cubique de Baum et Sweet. Ces travaux ont mis en
lumi\`ere un sous ensemble de s\'eries formelles alg\'ebriques, obtenues
comme points fixes de la compos\'ee d'une homographie \`a coefficients
entiers (polyn\^{o}mes) avec le morphisme de Frobenius ; ces s\'eries sont alors appel\'ees hyperquadratiques.
Le d\'eveloppement en fraction continue a pu \^{e}tre
donn\'e explicitement (voir \cite{mona2000,millsrobbins,schmidt} pour plus de r\'ef\'erences) pour de nombreux exemples de ces s\'eries formelles.
\par Nous avons observ\'e que pour tout nombre premier $p$ et $r=p^t$,
l'\'equation 
$$TX^{r+1}+X-T=0$$
 a une unique solution dans le corps
$\F_p((T^{-1}))$. Pour $r=p=2$ cette solution est la cubique de Baum
et Sweet. Dans cet article, nous donnons le d\'eveloppement en fraction
continue de cette solution pour $r>2$ (voir la partie \ref{resultats}). On peut par ailleurs remarquer
que si l'on remplace $r$ par $2$ dans les formules obtenues, alors le
d\'eveloppement obtenu est impropre (une sous-suite de quotients
partiels tend vers 0 dans $\F_p((T^{-1}))$). Cependant, cette
expression, lorsqu'elle est tronqu\'ee et rendue propre, donne le d\'eveloppement 
qui a \'et\'e obtenu dans \cite{millsrobbins}.
\par Pour obtenir ce d\'eveloppement nous utilisons une m\'ethode d\'ej\`a
utilis\'ee par Lasjaunias dans \cite{ffa2008}, qui, bien que proche de l'algorithme
de Mills et Robbins, en diff\`ere un peu. Pour illustrer cette m\'ethode,
nous l'appliquons dans un premier temps \`a un autre exemple de s\'erie formelle hyperquadratique, \`a coefficients dans $\F_p$.
 Cet exemple, tr\`es c\'el\`ebre, a \'et\'e introduit par
Mahler \cite{mahler} dans un article fondateur sur l'approximation diophantienne
dans les corps de fonctions.
\par Nous rappelons bri\`evement les notations utilis\'ees. Dans ce texte,
$p$ est un nombre premier,  $\F_p$ d\'esigne le
corps fini \`a $p$ \'el\'ements, $\F_p[T]$, $\F_p(T)$ et $\F_p((T^{-1}))$ sont,
respectivement, l'anneau des polyn\^{o}mes, le corps des fonctions
rationnelles et le corps des s\'eries formelles (en $1/T$) de la
variable $T$ sur $\F_p$. Ainsi
$$\F_p((T^{-1}))=\lbrace 0\rbrace \cup \left\{ \sum_{k\leq k_0}u_kT^k,k_0\in \mathbb Z,
u_k\in \F_p, u_{k_0}\neq 0 \right\}.$$
Ce corps de s\'eries formelles est muni d'une valeur absolue ultram\'etrique d\'efinie par $|\alpha|=|T|^{k_0}$ et $|0|=0$, o\`u $|T|$ est un r\'eel fix\'e strictement sup\'erieur \`a $1$. De plus, il est connu que $\F_p((T^{-1}))$ est le complet\'e de $\F_p(T)$ pour cette valeur absolue.

Dans la suite,  $r$ est une puissance de $p$,  $r=p^t$, avec $t\geq 1$ entier. Le morphisme de Frobenius d\'efini dans $\F_p((T^{-1}))$ est not\'e
$\alpha \mapsto \alpha^r$. Une
s\'erie formelle, $\alpha \in \F_p((T^{-1}))$, est dite hyperquadratique
si l'on a $\alpha =f(\alpha^r)$ o\`u $f$ est une homographie \`a
coefficients dans $\F_p[T]$. 
\par Tout \'el\'ement $\alpha \in \F_p((T^{-1}))$ a un
d\'eveloppement en fraction continue (infini si $\alpha$ n'est pas une fraction rationnelle)
que l'on notera
$$\alpha =[a_0,a_1,a_2,\dots,a_{n},\alpha_{n+1}]$$
o\`u les $a_i \in \F_p[T]$ (avec $\deg(a_i)>0$ pour $i>0$) sont appel\'es
les quotients partiels et les $\alpha_{i}\in \F_p((T^{-1}))$ sont les
quotients complets.

\vskip 1 cm

\section{M\'ethode employ\'ee et exemple de Mahler}

Dans cette partie nous pr\'esentons le raisonnement sur lequel les preuves reposent.
Nous commen\c{c}ons par un lemme \'el\'ementaire concernant les fractions continues. Une courte d\'emonstration en est donn\'ee dans l'article \cite{ffa2008}.

Nous rappelons d\'ej\`a la notation suivante. Soit $P/Q \in \F_p(T)$ tel que $P/Q:=[a_1,a_2,\ldots,a_n]$. Pour tout $x\in \F_p(T)$, nous noterons $$\Big[ [a_1,a_2,\ldots,a_n] ,x \Big]:=\frac{P}{Q}+\frac{1}{x}.$$

\begin{lemme} \label{lemmegen} Soient $a_1,\ldots, a_n, x\in \F_p(T)$. On a la relation suivante :
 \begin{equation}\nonumber \Big[ [a_1,a_2,\ldots,a_n] ,x \Big]=[a_1,a_2,\ldots,a_n,x'],\end{equation}
 avec   \begin{equation} \label{lgen} x'=f_nx+g_n,\end{equation}
 o\`u les $f_n,g_n$ sont des \'el\'ements de $\F_p(a_1,a_2,\ldots,a_n)$ (voir \cite{ffa2008}, page 330).
\end{lemme}

A l'exception du paragraphe \ref{contexte}, nous utilisons ce lemme uniquement dans les cas $n=2$ et $n=3$, qui s'\'enoncent comme suit.

\begin{lemme} \label{lemmepart}Soient $a_1,a_2,a_3,x\in \F_p(T)$. On a les relations suivantes :
 \begin{equation}\nonumber \Big[ [a_1,a_2] ,x \Big]=[a_1,a_2,y],
\text{ o\`u  } y=-a_2^{-2}x-a_2^{-1},
\end{equation}

\begin{equation}\nonumber \Big[ [a_1,a_2,a_3] ,x\Big]=[a_1,a_2, a_3, y], 
\text{ o\`u  } y=(a_2a_3+1)^{-2}x-a_2(a_2a_3+1)^{-1}.
\end{equation}
\end{lemme}

\subsection{Premier exemple}

Nous allons \`a pr\'esent d\'ecrire la suite de quotients partiels d'un ensemble de s\'eries formelles v\'erifiant un certain type d'\'equation. Comme corollaire, nous obtenons le d\'eveloppement en fraction continue de la s\'erie de Mahler :
\begin{equation}\label{mah} \Theta_r=1/T+1/T^{r}+1/T^{r^2}+\cdots+1/T^{r^k}+\cdots \in \F_p((T^{-1})). \end{equation}
 On peut remarquer que $\Theta_r$ est une s\'erie alg\'ebrique de degr\'e $r$ v\'erifiant l'\'equation 
\begin{equation}\label{maheq}  Tz^r-Tz+1=0.\end{equation}
Il convient de noter que le d\'eveloppement en fraction continue de $\Theta_r$, donn\'e dans le corollaire \ref{dfcmahler} est d\'ej\`a connu, m\^{e}me s'il n'a jamais \'et\'e pr\'esent\'e sous cette forme. En effet, il est expos\'e dans \cite{mona2000} (p. 215) et peut aussi \^{e}tre d\'eduit de travaux plus anciens de Shallit sur les fractions continues de certains nombres r\'eels \cite{shallit}.

\begin{theo}\label{mahlergen}Soit $p$ un nombre premier et $r=p^t$, $t\geq 1$, avec $r>2$. Soit $\ell\in \N$, $\ell\geq 1$ et soit $(a_1,a_2,\ldots,a_l)$ un $\ell$-uplet de polyn\^{o}mes dans $F_p[T]$, avec $a_i(T)\in T\F_p[T]$, pour tout $i$ impair, $1\leq i \leq l$. Si $z$ est la fraction continue  $z=[a_1,a_2,\ldots,a_{\ell},z_{\ell+1}]$ v\'erifiant l'\'equation :
\begin{equation}\label{mhgen} z^r=-T^2z_{\ell+1}-T, 
\end{equation}
alors  la suite de quotients partiels de $z$,  $(a_n)_{n\geq l+1} \in (\F_p[T])^{\N}$ est d\'efinie pour $k\geq 0$  par :
$$\begin{array}{ccccccc}a_{\ell+4k+1}&=&\displaystyle-\frac{a_{2k+1}^r}{T^2}, &\;&  a_{\ell+4k+2}&=&-T,\\
 a_{\ell+4k+3}&=&a_{2k+2}^r, &\;& a_{\ell+4k+4}&=&T. \\
\end{array}$$
\end{theo}

\begin{remarque} L'existence de la fraction continue v\'erifiant (\ref{mhgen}) d\'ecoule du th\'eor\`eme 1 de l'article \cite{ffa2008}.
\end{remarque}
Revenons maintenant \`a la s\'erie $\Theta_r$. On pose $y:=1/\Theta_r$ et $y:=[a_1,a_2,\ldots,a_n,\ldots]$. D'apr\`es (\ref{mah}), pour $r>2$, on a
$$\left|T-\frac{1}{\Theta_r}\right|=\frac{1}{|T^{r-1}\Theta_r|}=\frac{1}{|T|^{r-2}}<1,$$
et par cons\'equent $y=T+1/y_2=[T,y_2]$. D'apr\`es (\ref{maheq}) on a 
$$\frac{T}{y^r}=\frac{T}{y}-1=-\frac{1}{Ty_2+1},$$
et donc $y_2$ v\'erifie la relation :
\begin{equation}\label{mahler1} y^r=-T^2y_2-T.\end{equation}

En remarquant que l'\'equation (\ref{mahler1}) est un cas particulier de l'\'equation (\ref{mhgen}), o\`u $\ell=1$ at $a_1=T$, nous en d\'eduisons directement le corollaire suivant.
\begin{corollaire} \label{dfcmahler}On a
  $\Theta_r=[0,a_1,a_2,\dots,a_n,\dots]$ o\`u la suite $(a_i)_{i\geq 1}$ est
  d\'efinie par r\'ecurrence, pour $k\geq 0$, par :
$$\begin{array}{ccccccc}a_{4k+1}&=&T,  &\;& a_{4k+2}&=&-a_{2k+1}^r/T^2,\\
   a_{4k+3}&=&-T,  &\;& a_{4k+4}&=&a_{2k+2}^r. \\
\end{array}$$
\end{corollaire}

\begin{proof}[D\'emonstration du th\'eor\`eme \ref{mahlergen}] Nous partons de la relation :
\begin{equation}\nonumber z^r=-T^2z_{\ell+1}-T,\end{equation}
qui peut \^{e}tre \'ecrite aussi sous la forme $[a_1^r,z_2^r]=-T^2z_{\ell+1}-T$, ou encore :
$$\Big[[-\frac{a_1^r}{T^2},-T],-T^2z_2^r\Big]=z_{\ell+1}.$$
%$$-\frac{a_1^r}{T^2}-\frac{1}{T}-\frac{1}{T^2z_2}=z_{\ell+1}.$$
En utilisant le lemme \ref{lemmepart} et en tenant en compte du fait que $a_1$ est divisible par $T$, nous en d\'eduisons les relations suivantes :
$$z_{\ell+1}=[-\frac{a_1^r}{T^2},-T, z'], \text{ avec } z'=z_2^r+T^{-1}.$$
Puisque $|z'|>1$, on en d\'eduit que $z'=z_{l+3}$. On a donc
$$a_{\ell+1}=-\frac{a_1^r}{T^2}, \, a_{\ell+2}=-T \text{ et }z_{l+3}=z_2^r+\frac{1}{T}.$$\\
Nous appliquons de nouveau le m\^{e}me raisonnement %\`a l'\'equation (\ref{mhgen2}) 
et nous obtenons : $$z_{l+3}=z_2^r+\frac{1}{T}=[a_2^r+\frac{1}{T},z_3^r]=[a_2^r,T,z''],$$
avec $z''=-T^{-2}z_3^r-T^{-1}$. Puisque $r>2$, on a $|z''|>1$ et alors, par identification, $z''=z_{l+5}$. On a donc
$$a_{\ell+3}=a_2^r,\, a_{\ell+4}=T \text{ et } z_{l+5}=-T^{-2}z_{3}^r-T^{-1}.$$

En r\'esum\'e, %(\ref{mh2}) et (\ref{mh3}) impliquent :
\begin{equation}\nonumber z_{\ell+1}=[-\frac{a_1^r}{T^2},-T,a_2^r,T, z_{\ell+5}].\end{equation}
Plus g\'en\'eralement, par une simple r\'ecurrence sur $k$, on obtient que :
\begin{equation}\nonumber z_{2k+1}^r=-T^2z_{\ell+4k+1}-T. \end{equation}
Puisque $a_{2k+1}\in T\F_p[T]$, nous obtenons comme pr\'ec\'edemment :
\begin{equation}\nonumber z_{\ell+4k+1}=[-\frac{a_{2k+1}^r}{T^2},-T,a_{2k+2}^r,T, z_{\ell+4k+5}],\end{equation}
ce qui termine la d\'emonstration.
\end{proof}

\subsection{Le contexte g\'en\'eral}\label{contexte}
Dans cette partie nous \'enon\c{c}ons le raisonnement g\'en\'eral et les notations utilis\'ees dans les preuves qui suivent. Ceux-ci restent proches de ceux utilis\'es par Mills et Robbins dans \cite{millsrobbins}.

 Soit $(P,Q,R)\in \F_p[T]^3$ et soient $m,n$ deux entiers strictement positifs, $m<n$. On dit que $z$ satisfait une relation du type $(P,Q,R,m,n)$ si on a :
\begin{equation}\label{pqr} Pz_m^{r}=Qz_n+R.\end{equation}
%Dans ce qui suit, nous allons d\'ecrire la m\'ethode qui nous permet de retrouver les quotients partiels d'une s\'erie formelle qui satisfait une \'equatio du type $(P,Q,R)$.
Dans la suite, nous consid\'erons une s\'erie $z=[a_1,a_2,\ldots]\in \F_p((T^{-1}))$ satisfaisant (\ref{pqr}), avec un triplet $(P,Q,R)$ bien choisi. 
%Nous supposons connus les $n-1$ premiers quotients partiels. 
En fait, il est probable que cette relation soit vraie pour presque toutes les s\'eries formelles hyperquadratiques, mais \`a notre connaissance, aucun r\'esultat g\'en\'eral ne le confirme (le lecteur peut consulter \cite{ffa2008}, page 333, pour quelques commentaires \`a ce sujet).

On suppose connus les $n-1$ premiers coefficients partiels : $a_1,a_2,\ldots,a_{n-1}$ (qui peuvent \^{e}tre vus comme une ``donn\'ee de d\'epart'').
La relation (\ref{pqr}) implique :
$$Pa_m^r+\frac{P}{z_{m+1}^r}=Qz_n+R,$$
car, par d\'efinition, $z_m=[a_m, z_{m+1}]$.
Ainsi, nous obtenons :
\begin{equation}\label{eq3} \frac{Pa_m^r-R}{Q}+\frac{P}{Qz_{m+1}^r}=z_n.\end{equation}
Puisque $$\frac{Pa_m^r-R}{Q}\in \F_p(T),$$ cette expression a un d\'eveloppement en fraction continue qui est fini. Il existe ainsi des polyn\^{o}mes $\lambda_1,\lambda_2,\ldots, \lambda_\ell\in \F_p[T]$, que l'on peut calculer, tels que :
$$\frac{Pa_m^r-R}{Q}=[\lambda_1,\lambda_2,\ldots, \lambda_{\ell}].$$ 
La relation (\ref{eq3}) devient :
\begin{equation}\nonumber\Big[[\lambda_1,\lambda_2,\ldots, \lambda_{\ell}], \frac{Qz_{m+1}^r}{P}\Big]=z_n,\end{equation}
et le lemme \ref{lemmegen} implique alors que :
$$z_n=[\lambda_1,\lambda_2,\ldots, \lambda_{\ell}, z'].$$
A ce moment, si $|z'| >1 $, nous pouvons  d\'ej\`a identifier $a_n=\lambda_1,a_{n+1}=\lambda_2, \ldots, a_{n+l-1}=\lambda_l$ et $z_{n+l}=z'$.\\
Par le lemme \ref{lemmegen} on d\'eduit donc :
$$z_{n+\ell}=f_{\ell}\frac{Qz_{m+1}^r}{P}+g_{\ell},$$
o\`u $f_{\ell}$ et $g_{\ell}$ appartiennent \`a $\F_p(T)$ (voir la formule (\ref{lgen})).
Il existe donc trois polyn\^{o}mes $P_1,Q_1,R_1\in \F_p[T]$, qu'on peut d\'eterminer explicitement, tels que :
\begin{equation}\nonumber  P_1z_{m+1}^r= Q_1 z_{n+\ell}+R_1. \end{equation}

En r\'esum\'e, la connaissance de $P,Q,R$ et $a_m$ nous permet de d\'eterminer 
les $\ell$ nouveaux quotients partiels : $a_n, a_{n+1}, \ldots, a_{n+\ell-1}$ et une nouvelle \'equation :
$$P_1z_{m+1}^r=Q_1z_{n+\ell}+R_1.$$
Dans la suite, nous noterons ce raisonnement par :
$$(P,Q,R,m,n:a_m)\rightarrow (P_1,Q_1,R_1,m+1,n+\ell : a_{n},a_{n+1},\ldots,a_{n+\ell-1}).$$
%Le proc\'ed\'e peut �tre it\'er\'e et nous pouvons obtenir tous les quotients partiels en fonction des pr\'ec\'edents.
Plus g\'en\'eralement, si $X:=(P,Q,R,m,n)$ et $Y:=(P_{\ell},Q_{\ell},R_{\ell},m+\ell,n+k)$, nous noterons :
\begin{equation}\label{ctx}(X:a_m, a_{m+1},\ldots,a_{m+\ell-1})\rightarrow (Y: a_n,a_{n+1},\ldots, a_{n+k-1})\end{equation}
pour d\'esigner la suite des raisonnements suivants :
\begin{displaymath}
\left\{ \begin{array}{ll}
(X:a_m) &\rightarrow (X_1:a_n,a_{n+1},\ldots, a_{n+k_1-1})\\
(X_1:a_{m+1}) &\rightarrow  (X_2:a_{n+k_1},a_{n+k_1+1}, \ldots a_{n+k_2-1})\\
 &\quad \vdots \\
(X_{\ell-1}:a_{m+\ell-1}) &\rightarrow  (Y:a_{n+k_{\ell-1}},a_{n+k_{\ell-1}+1},\ldots,a_{n+k-1}).
\end{array}  \right.
\end{displaymath}
o\`u $X_i: = (P_i,Q_i,R_i,m+i,n+k_i)$, pour $1\leq i \leq \ell-1$.
La nouvelle \'equation va donc relier  $z_{m+\ell}$ et $z_{n+k}$ de la mani\`ere suivante :
$$P_{\ell}z_{m+\ell}^r=Q_{\ell}z_{n+k}+R_{\ell}.$$
Nous remarquons que, lorsque $l=n-m$, la relation (\ref{ctx}) s'\'ecrit :
$$(P,Q,R,m,n:a_m, a_{m+1},\ldots,a_{n-1})\rightarrow (P_{n-m},Q_{n-m},R_{n-m}: a_n,a_{n+1}\ldots, a_{n+k-1}),$$
ce qui signifie que la connaissance de $a_m,a_{m+1},\ldots,a_{n-1}$ nous permet de calculer les $k$ quotients partiels suivants, \`a savoir $a_n,a_{n+1},\ldots, a_{n+k-1}$. Nous pouvons alors appliquer de nouveau le m\^{e}me raisonnement, \`a une \'equation du type $(P_{n-m},Q_{n-m},R_{n-m},m+n-1,n+k)$ ayant comme ``donn\'ee de d\'epart'' les quotients partiels $a_n,a_{n+1},\ldots, a_{n+k-1}$. L'it\'eration de ce proc\'ed\'e permet ainsi d'obtenir tous les quotients partiels de la s\'erie $z$.

Soient $W,W'$ deux suites finies de polyn\^{o}mes \`a coefficients dans $\F_p$. On note $|W|$ la longueur de $W$, c'est-\`a-dire le nombre de ses termes. Alors la relation (\ref{ctx}) peut \^{e}tre \'ecrite aussi :
$$(P,Q,R,m,n : W)\rightarrow (P_{|W|},Q_{|W|},R_{|W|},m+|W|, n+ |W'|:W').$$
Pour all\'eger les notations, nous nous permettons parfois d'\'ecrire :
$$(P,Q,R,m,n : W)\rightarrow (P_{|W|},Q_{|W|},R_{|W|}:W').$$

\section{G\'en\'eralisation de la cubique de Baum et Sweet}\label{resultats}

Avant d'\'enoncer notre r\'esultat principal, nous allons tout d'abord introduire quelques d\'efinitions.

\begin{defi}  On d\'efinit une suite de polyn\^{o}mes \`a coefficients dans $\F_p$, $(\Gamma_k)_{k\geq 1}$, de la fa\c{c}on suivante.
On pose $\Gamma_1:=-T, T^r, T$. Pour $k >1$, $\Gamma_k$ est d\'efini r\'ecursivement par : 
$$\Gamma_k:=a_1,a_2,\ldots,a_{2^{k+1}-1}\text { et } \Gamma_{k+1}:=b_1,b_2,\ldots,b_{2^{k+2}-1},$$
o\`u $$b_{2i-1}=(-1)^{i+k}T \quad \text{ pour } 1\leq i\leq 2^{k+1}$$
et $$b_{4i}=a_{2i}^r/T^2 \text{ pour } 1\leq i\leq 2^{k}-1, \quad b_{4i-2}=-a_{2i-1}^r \text{ pour } 1\leq i\leq 2^{k}.$$
\end{defi}

\begin{defi} On d\'efinit la suite des polyn\^{o}mes \`a coefficients dans $\F_p$, $(\Lambda_k)_{k\geq 1}$, de la fa\c{c}on suivante.  On pose
  $\Lambda_1:=T+1,T-1$, $\Lambda_2:=T,-T^r+1,-T $. Pour $k\geq 3$, $\Lambda_k$ est d\'efini r\'ecursivement par :
 $$\Lambda_k:=\Lambda_{k-2},-T^{\lambda_{k-1}}, \Gamma_{k-2},$$
 o\`u $(\lambda_k)_{k\geq 1}$ est d\'efini comme suit :
 $$\lambda_1=r,\quad  \lambda_{k+1}=r\lambda_k-2.$$
\end{defi}
Etant donn\'ee une suite $W=a_1,\ldots,a_m$ \`a valeurs dans $\F_p[T]$, on note 
$-W:=-a_1,\ldots,-a_m$ et $\overline{W}:=a_m, a_{m-1}, \ldots, a_1$.
\begin{defi}  On d\'efinit la suite des polyn\^{o}mes \`a coefficients dans $\F_p$, $(\Omega_k)_{k\geq 1}$, de la fa\c{c}on suivante.
On pose $\Omega_1:=-T^{\omega_1}$. Pour $k\geq 2$, $\Omega_k$ est d\'efini r\'ecursivement par :
\begin{equation}\label{ommega}\Omega_k:= \Omega_{k-1}, \Lambda_{k-1}, -T^{\omega_k}, -\overline{\Lambda}_{k-1},  -\overline{\Omega}_{k-1},\end{equation}
o\`u $(\omega_k)_{k\geq 1}$ est d\'efini comme suit :
 $$\omega_1=r-2, \quad \omega_{k+1}=r\omega_k-2.$$
Nous notons $\Omega_{\infty}$ la suite infinie commen\c{c}ant par $\Omega_k$, pour tout $k\geq 1$.
 \end{defi}

\bigskip

Avec les notations ci-dessus, nous pr\'esentons maintenant notre r\'esultat principal.
\begin{theo} \label{bsgen} Soit $p$ un nombre premier et $r=p^t$, o\`u $t\geq 1$, avec $r>2$.  L'\'equation 
\begin{equation} \nonumber TX^{r+1}+X-T=0
\end{equation}
a une unique racine dans le corps $\F_p((T^{-1}))$ dont le d\'eveloppement en fraction continue est :
$$[1,-T-1,\Omega_{\infty}].$$ 
\end{theo}
\bigskip

\subsection{D\'emonstration du th\'eor\`eme \ref{bsgen}}

Tout d'abord, nous montrons que l'\'equation :
 \begin{equation}\label{bsggen} TX^{r+1}+X-T=0
\end{equation}
a une unique solution dans le corps $\F_p((T^{-1}))$. \\
En effet, de (\ref{bsggen}), on d\'eduit que :
\begin{equation}\label{genbs}X=\frac{T}{TX^r+1}.\end{equation}
Nous consid\'erons maintenant l'application $f(X)=T/(TX^r+1)$ d\'efinie sur $\F_p((T^{-1}))$ \`a valeurs dans  $\F_p((T^{-1}))$. Il n'est pas difficile de voir que $f$ est une application strictement contractante (puisque $r>2$) et nous savons que $\F_p((T^{-1}))$ est un espace complet pour la distance ultram\'etrique usuelle. Ainsi, par utilisation du th\'eor\`eme du point fixe, l'\'equation $f(X)=X$ a une unique solution dans $\F_p((T^{-1}))$. Celle-ci est donc l'unique solution de l'\'equation (\ref{bsggen}), que nous noterons dans la suite $BS_r$.\\
Soit $z$ la s\'erie d\'efinie par :
\begin{equation}\label{bsr}BS_r=1+\frac{1}{(-T-1)+1/z}.\end{equation}
D'apr\`es (\ref{genbs}) et (\ref{bsr}) on a :
\begin{equation}\label{bsr2} 
z=\frac{(-T^r+T-1)z^r+1}{T^2z^r},
\end{equation}
ce qui entra\^{i}ne 
$$\left| z-\frac{-T^r+T-1}{T^2}\right|=\frac{1}{|T^2z^r|}< \frac{1}{|T^4|}.$$
Autrement dit $(-T^r+T-1)/T^2$ est une r\'eduite de $z$, donc les 3 premiers quotients partiels de $z$ sont 
$-T^{r-2}, T+1$ et $T-1$. Ainsi 
$z=[-T^{r-2}, T+1,T-1,z_4]$ et d'apr\`es (\ref{bsr2}) on obtient la relation suivante :   
\begin{equation}\nonumber z^r=T^2z_4+(T+1).\end{equation}

Le th\'eor\`eme \ref{bsgen} est alors une cons\'equence directe de proposition ci-dessous.

\begin{proposition}\label{bsgenprop}Soit $z$ la fraction continue infinie $z:=[-T^{r-2},T+1,T-1,z_4]$ satisfaisant :
\begin{equation}\label{eq1} z^r=T^2z_4+(T+1).
\end{equation}
Alors la suite de quotients partiels de $z$ est $\Omega_{\infty}$.
\end{proposition}

\begin{notation} On dit qu'une \'equation est du type $A_1$ si $P=1$, $Q=T^2$ et $R=T+1$ et on note $A_1:=(1,T^2, T+1)$. De la m\^{e}me mani\`ere, nous allons d\'efinir les \'equations des types suivants :
\begin{align*}
A_2:&=(1,T^2,-T+1),\\
A_3:&=(T,-T,-1),\\
A_4:&=(1,T^2,-T),\\
A_5:&=(T,-T,1),\\
A_6:&=(1,T^2,T).
\end{align*}
%Par exemple, la relation (\ref{eq1}) est du type $A_1:=(1, T^2, T+1)$.
\end{notation}

Dans ce qui suit, nous utiliserons les notations d\'ecrites dans la partie \ref{contexte}.
\begin{lemme}\label{formules} Soient $m,n \in \mathbb N$ tels que $m<n$ et soit $a\in \F_p[T]$. On a les relations suivantes :
\begin{align*}
(A_1,m,n:a)&\rightarrow \left(A_2,m+1,n+3 :\frac{a^r}{T^2}, -T+1,-T-1\right) \text{ si }a\equiv 0 [T],\\
(A_1,m,n:a)&\rightarrow \left(A_5,m+1,n+2: \frac{(a-1)^r}{T^2}, -T\right) \text{ si }a\equiv 1 [T],\\
(A_2,m,n:a)&\rightarrow \left(A_1,m+1,n+3:\frac{a^r}{T^2}, T+1, T-1\right) \text{ si }a\equiv 0 [T],\\
(A_2,m,n:a)&\rightarrow \left(A_3,m+1,n+2:\frac{(a-1)^r}{T^2}, T \right)\text{ si }a\equiv 1 [T],\\
(A_3,m,n:a)&\rightarrow \left(A_4,m+1,n+2: -a^r, -T\right) \text{ pour tout }a \in \F_p[T],\\
(A_4,m,n:a)&\rightarrow \left(A_3,m+1,n+2:\frac{a^r}{T^2}, T \right) \text{ si }a\equiv 0 [T],\\
(A_4,m,n:a)&\rightarrow \left(A_2,m+1,n+3:\frac{(a+1)^r}{T^2}, T+1,T-1\right) \text{ si }a\equiv -1 [T],\\
(A_5,m,n:a)&\rightarrow \left(A_6,m+1,n+2:-a^r, T\right) \text{ pour tout }a \in \F_p[T],\\
(A_6,m,n:a)&\rightarrow \left(A_5,m+1,n+2:\frac{a^r}{T^2}, -T \right) \text{ si } a\equiv 0 [T],\\
(A_6,m,n:a)&\rightarrow \left(A_2,m+1,n+3:\frac{(a+1)^r}{T^2}, -T+1,-T-1: \right) \text{ si }a\equiv -1 [T].\\
\end{align*}
\end{lemme}

\begin{proof}[D\'emonstration]

Nous allons prouver le premier cas, c'est \`a dire le cas o\`u $z$ satisfait une relation du type $A_1$, avec $a\equiv 0 [T]$ :
 $$z_{m}^r=T^2z_{n}+(T+1).$$
Cette relation s'\'ecrit aussi sous la forme :
 $[a^r, z_{m+1}^r]=T^2z_{n}+(T+1)$ ou bien
$$\frac{a^r}{T^2}-\frac{T+1}{T^2}+\frac{1}{T^2z_{m+1}}=z_{n}.$$
Puisque $a$ est divisible par $T$, nous obtenons : $$\Big[[\frac{a^r}{T^2},-T+1,-T-1],T^2z_{m+1}\Big]=z_{n}.$$ 
En appliquant  le lemme \ref{lemmepart}, on en d\'eduit que 
 $$z_{n}=[\frac{a^r}{T^2},-T+1,-T-1, z_{n+3}]$$ et
$$z_{m+1}^r=T^2z_{n+3}+(-T+1).$$
Ainsi nous obtenons les nouveaux quotients partiels $a_n=\frac{a^r}{T^2},a_{n+1}=-T+1,a_{n+2}=-T-1$ et une \'equation du type $A_2:=(1,T^2,-T+1)$.\\
Les autres cas se d\'eduisent de mani\`ere analogue, en appliquant le raisonnement pr\'ec\'edent et, en particulier, le lemme \ref{lemmepart}.
\end{proof}

\begin{notation} %\'etant donn\'ee une suite $W=a_1,a_2,\ldots,a_m$  \`a valeurs dans $\F_p[T]$, on note $-W:=-a_1,-a_2,\ldots,-a_m$ et $\overline{W}:=a_m, a_{m-1}, \ldots, a_1$.
% Soit $R:=a_1,a_2,\ldots,a_m$ une suite \`a valeurs dans $\F_q[T]$. 
Soient $i, j \in \N^*$, $i,j \leq m$ et $W=a_1,a_2,\ldots,a_{m}$. On notera
$$ ^{(i)}W := a_{i+1}, a_{i+2},\ldots, a_m \text{ et }W^{(j)}:=a_1,a_2,\ldots,a_{m-j}.$$ Lorsque $i+j<m-1$, on notera
$$^{(i)}W^{(j)}:=a_{i+1},\ldots,a_{m-j}.$$
\end{notation}

La proposition \ref{bsgenprop} est une cons\'equence imm\'ediate de la proposition \ref{Omega}.

\begin{proposition}\label{Omega} Pour tout  $k\in \N^*$, on a la relation suivante :
\begin{equation}\nonumber (A_1,1,4:\Omega_k) \rightarrow \left(A_2,1+|\Omega_k|, |\Omega_{k+1}|: \,^{(3)} \Omega^{(1)}_{k+1}\right). \end{equation}
\end{proposition}

\begin{remarque} Soit $k\in \N$. Nous notons $\ell_k$ la longueur du mot fini $\Omega_k$. La proposition pr\'ec\'edente peut \^{e}tre traduite de la mani\`ere suivante. On suppose connus les $\ell_k$ premiers  quotients de $z$. Alors, si on  applique $\ell_k$ fois le proc\'ed\'e \'enonc\'e dans le paragraphe \ref{contexte} \`a l'\'equation (\ref{eq1}) nous obtenons :
$$z_4=[^{(3)} \Omega^{(1)}_{k+1},z_{\ell_{k+1}}]$$ 
et la nouvelle relation sera :
$$z_{\ell_k+1}=T^2 z_{\ell_{k+1}}+(T+1).$$
\end{remarque}

Pour d\'emontrer la proposition \ref{Omega}, nous allons utiliser les lemmes suivants.

\begin{lemme}\label{Gamma} Soient $k,m,n$ des entiers strictement positifs, $m<n$. On a les relations suivantes :
\begin{itemize}
\item si $k$ est pair, alors :
\begin{align*}
(A_5,m,n&:\Gamma_k)\rightarrow (A_6:\,^{(1)}\Gamma_{k+1}),\\
%(E&:\overline{\Gamma}_n )\rightarrow (T^{-1}\overline{\Gamma}_{n+1}:F)\\
(A_5,m,n&:-\overline{\Gamma}_k) \rightarrow (A_6:\,^{(1)}(-\overline{\Gamma}_{k+1})) \, ;
\end{align*}
\item si $k$ est impair, alors :
\begin{align*}
(A_3,m,n&:\Gamma_k)\rightarrow (A_4:\,^{(1)}\Gamma_{k+1}),\\
%(C&:\overline{\Gamma}_n) \rightarrow ((-T)^{-1}\overline{\Gamma}_{n+1}:D)\\
(A_3,m,n&:-\overline{\Gamma}_k) \rightarrow (A_4:\,^{(1)}(-\overline{\Gamma}_{k+1})).
\end{align*}
\end{itemize}
\end{lemme}

\begin{proof}[D\'emonstration]
Tout d'abord, on remarque que, pour tout $k \in \N$, les termes de $\Gamma_k$ sont des polyn\^{o}mes divisibles par $T$.\\ 
Soient $a,b \in T\F_p(T)$. Alors, \`a l'aide du lemme \ref{formules} on peut d\'eduire :
\begin{equation}\nonumber
(A_5,m,n: a,b) \rightarrow (A_5, m+2,n+4 : -a^r,T,b^r/T^2,T).
\end{equation}
Plus g\'en\'eralement, si $a_1,a_2,\ldots,a_{2i} \in T\F_p[T]$ alors :
\begin{equation}\nonumber
(A_5,m,n: a_1,a_2\ldots,a_{2i}) \rightarrow (A_5 : -a_1^r,T,a_2^r/T^2,T,\ldots,-a_{2i-1}^r,T,a_{2i}^r/T^2,T).
\end{equation}
Soit $k\in \N$ et $\Gamma_k:=a_1,a_2,\ldots,a_{2^{k+1}-1}$. La suite $\Gamma_k$ a un nombre impair d'\'el\'ements ; ainsi, nous devons appliquer le lemme \ref{formules} \`a son dernier terme aussi. On obtient :
\begin{equation}\nonumber
(A_5,m,n: \Gamma_k) \rightarrow (A_6 : -a_1^r,T,a_2^r/T^2,T,\ldots,-a_{2i-1}^r,T,a_{2i}^r/T^2,T,-a_{2i+1}^r,T).
\end{equation}
Dans le cas o\`u $k$ est pair, le premier terme de $\Gamma_{k+1}$ est $b_1:=-T$. Donc, par d\'efinition, $$^{(1)}\Gamma_{k+1} =b_2,\ldots,b_{2^{k+2}-1}=-a_1^r,T,\ldots,-a_{2i-1}^r,T,a_{2i}^r/T^2,T,-a_{2i+1}^r,$$ ce qui co\"{i}ncide avec notre r\'esultat.\\
Les autres cas se d\'emontrent de mani\`ere analogue.
%La preuve est directe ; il suffit d'utiliser la d\'efinition de la suite $(\Gamma_k)_{k\geq 1}$ et le lemme \ref{formules}.
\end{proof}

\begin{lemme}\label{Lambda} Soient $k,m,n$ des entiers strictement positifs, $m<n$. On a les relations suivantes :
\begin{itemize}
\item si $k$ est pair, alors : 
\begin{align} 
\label{aa}(A_2,m,n&:\Lambda_k)\rightarrow (A_6: T^{r-2},\Lambda_{k+1}),\\
\label{bb}(A_5,m,n&:-\overline{\Lambda}_k)\rightarrow (A_1:\,^{(1)}(-\overline{\Lambda}_{k+1}),T^{r-2},T+1,T-1).
\end{align}

\item si  $k$ est impair, alors : 
\begin{align}
\label{cc}(A_2,m,n&:\Lambda_k)\rightarrow (A_4:T^{r-2},\Lambda_{k+1}),\\
\label{dd}(A_3,m,n&:-\overline{\Lambda}_k)\rightarrow (A_1:\,^{(1)}(-\overline{\Lambda}_{k+1}),T^{r-2},T+1,T-1).
 \end{align}
\end{itemize}
\end{lemme}

\begin{proof}[D\'emonstration]
 Nous allons d\'emontrer ici la relation (\ref{aa}). Les relations (\ref{bb}), (\ref{cc}) et (\ref{dd}) peuvent \^{e}tre d\'emontr\'ees de mani\`ere analogue.
 On raisonne par r\'ecurrence sur $k$.\\
Nous commen\c{c}ons par le cas o\`u $k=2$. Par d\'efinition : $$\Lambda_2=T,-T^r+1,-T$$ et $$\Lambda_3=T+1,T-1,-T^{r^2-2}, -T, T^r, T.$$
En appliquant les formules du lemme \ref{formules}, on a :
\begin{align*}
(A_2,m,n: T)&\rightarrow (A_1,m+1,n+3 :T^{r-2},T+1,T-1)\\
(A_1,m+1,n+3:-T^r+1)&\rightarrow (A_5,m+2,n+5:-T^{r^2-2},-T)\\
(A_5,m+2,n+5:-T)&\rightarrow ( A_6:-T^r,T)
\end{align*}
 Par cons\'equent, (\ref{aa}) est prouv\'e pour $k=2$.\\
 Nous supposons \`a pr\'esent que $k$ est un entier pair, $k> 2$, et que la relation (\ref{aa}) est vraie pour $k-2$. Nous allons la montrer maintenant pour $k$.\\
 Par d\'efinition, $$\Lambda_k= \Lambda_{k-2}, -T^{\lambda_{k-1}}, \Gamma_{k-2}.$$
 Par l'hypoth\`ese de r\'ecurrence,
 $$(A_2,m,n:\Lambda_{k-2})\rightarrow (A_6,m+|\Lambda_{k-2}|,n+|\Lambda_{k-1}|+1 :T^{r-2},\Lambda_{k-1}).$$
 En appliquant les lemmes \ref{formules} et \ref{Gamma} nous avons aussi :
\begin{align*}
(A_6,m+|\Lambda_{k-2}|,n+|\Lambda_{k-1}|+1:  -T^{\lambda_{k-1}})&\rightarrow (A_5:-T^{\lambda_{k}},-T)\\
(A_5,m+|\Lambda_{k-2}|+1,n+|\Lambda_{k-1}|+3:\Gamma_{k-2})&\rightarrow (A_6:\,^{(1)}\Gamma_{k-1}).
\end{align*}
En r\'eunissant tous ces relations et en tenant en compte que, pour tous les $k$ pairs, $\Gamma_{k-1}$ commence par $-T$, on obtient le r\'esultat.
\end{proof}

\bigskip

\begin{proof}[D\'emonstration de la proposition \ref{Omega}]

 Nous allons prouver par r\'ecurrence sur $k$ que, pour tous $m,n$ tels que $m<n$, on a les relations suivantes : 
\begin{equation}\label{prefix1} (A_1,m,n:\Omega_k) \rightarrow \left( A_2:\,^{(3)} \Omega^{(1)}_{k+1}\right), \end{equation}
\begin{equation}\label{prefix2} (A_1,m,n:-\overline{\Omega}_k) \rightarrow \left(A_2:\,^{(3)} (-\overline{\Omega}_{k+1})^{(1)}\right). \end{equation}
Lorsque $m=1,n=4$, la relation (\ref{prefix1}) implique la proposition \ref{Omega}.
On commence par le cas o\`u $k=1$. On a :
\begin{align*} \Omega_1&=-T^{r-2}, -\overline{\Omega_1}=T^{r-2},\\
\Omega_2&=-T^{r-2},T+1,T-1,-T^{\omega_2},-T+1,-T-1,T^{r-2},\\
-\overline{\Omega}_2&=-T^{r-2},T+1,T-1,T^{\omega_2},-T+1,-T-1,T^{r-2}.
\end{align*}
D'autre part, le lemme \ref{formules} implique que pour tous $m,n$, $m<n$ :
\begin{align*}
(A_1,m,n&:\pm T^{r-2})\rightarrow(A_2:\pm T^{r(r-2)-2},-T+1,-T-1),
\end{align*}
et puisque $\omega_2=r(r-2)-2$ les relations (\ref{prefix1}) et (\ref{prefix2}) sont \'etablies pour $k=1$.\\
Soit $k > 1$. Nous supposons maintenant que (\ref{prefix1}) et (\ref{prefix2}) sont vraies pour $k-1$ et tous $m,n$, $m<n$.
A cause du lemme \ref{Lambda}, nous devons distinguer deux cas : le cas o\`u $k$ est pair et le cas o\`u $k$ est impair. Nous allons supposer maintenant que $k$ est pair.\\
Fixons $m,n \in \N$, $m<n$.

D'apr\`es (\ref{ommega}) on peut \'ecrire les deux relations suivantes : 
\begin{align*}
\Omega_k&= \Omega_{k-1}, \Lambda_{k-1},  -T^{\omega_k},-\overline{\Lambda}_{k-1}, -\overline{\Omega}_{k-1},\\
-\overline{\Omega}_{k}&=\Omega_{k-1}, \Lambda_{k-1}, T^{\omega_k}, -\overline{\Lambda}_{k-1},-\overline{\Omega}_{k-1},
\end{align*}
afin d'appliquer notre algorithme \`a chaque sous-suite qui appara\^{i}t dans les expressions de $\Omega_k$ et $-\overline{\Omega}_k$.\\
Soient $1 \leq i,i' \leq 6$ et $W,W'$ des suites finies de polyn\^{o}mes sur $\F_p$. Pour all\'eger l'\'ecriture, nous nous permettons d'utiliser la notation :
$$(A_i: W) \rightarrow (A_{i'}:W'),$$
les indices $m$ et $n$ \'etant sous-entendus.\\
Par l'hypoth\`ese de r\'ecurrence, on a :
\begin{align*}
(A_1&:\Omega_{k-1}) \rightarrow ( A_2:\,^{(3)}\Omega^{(1)}_{k}),\\
(A_1&:-\overline{\Omega}_{k-1}) \rightarrow \left(A_2:\,^{(3)} (-\overline{\Omega}_{k})^{(1)}\right).
\end{align*}
Les lemmes \ref{Lambda} et \ref{formules} impliquent :
\begin{align*}
(A_2&: \Lambda_{k-1})\rightarrow (A_4:T^{r-2},\Lambda_k),\\
(A_4& : \pm T^{\omega_{k}})\rightarrow (A_3:\pm T^{\omega_{k+1}},T),\\
(A_3&:-\overline{\Lambda}_{k-1})\rightarrow (A_1:\,^{(1)}(-\overline{\Lambda}_{k}),T^{r-2},T+1,T-1).
\end{align*}\\
En r\'esum\'e, on a :
\begin{align*}(A_1& : \Omega_{k}) \rightarrow (A_2:\, ^{(3)}\Omega^{(1)}_{k},T^{r-2}, \Lambda_k,  T^{\omega_{k+1}},T, \,^{(1)}(-\overline{\Lambda}_{k}),T^{r-2},T+1,T-1,\,^{(3)}(-\overline{\Omega}_{k})^{(1)}),\\
(A_1 &:\overline{\Omega}_{k}) \rightarrow (A_2: \,^{(3)}\Omega^{(1)}_{k},T^{r-2}, \Lambda_k,  -T^{\omega_{k+1}},T, \,^{(1)}(-\overline{\Lambda}_{k}),T^{r-2},T+1,T-1,\,^{(3)}(-\overline{\Omega}_{k})^{(1)}).
\end{align*}\\
Nous remarquons que pour tous les $k>1$, $\Omega_k$ commence par $-T^{r-2},T+1,T-1$ et il se termine par $-T+1,-T-1,T^{r-2}$. De m\^{e}me,  $-\overline{\Lambda}_k$ commence par $T$ (puisque $\Gamma_{k-2}$ se termine par $-T$). En combinant les relations pr\'ec\'edentes, nous en d\'eduisons  (\ref{prefix1}) et (\ref{prefix2}).

Le cas o\`u $k$ est impair se traite de mani\`ere analogue.
\end{proof}

\vspace{15mm}

Nous avons observ\'e que tout le raisonnement ci-dessus pour obtenir le d\'eveloppement en fraction continue de $BS_r$ est bas\'e sur le fait que le premier des 3 quotients partiels de d\'epart $(-T^{r-2},T+1,T-1)$ est divisible par $T$. Ainsi, nous pouvons obtenir un r\'esultat plus g\'en\'eral en rempla\c{c}ant $-T^{r-2}$ par un polyn\^{o}me $P$ arbitraire, divisible par $T$. En utilisant les m\^{e}mes notations qu'au d\'ebut de ce paragraphe, nous avons le th\'eor\`eme ci-dessous dont la preuve s'obtient comme pr\'ec\'edemment.

\begin{theo}
 Soit $P \in T\F_p[T]$. On d\'efinit la suite des polyn\^{o}mes \`a coefficients dans $\F_p$, $(\Omega_k(P))_{k\geq 0}$, de la fa\c{c}on suivante.

On pose $\Omega_1(P):=P$. Pour $k\geq 2$, $\Omega_k(P)$ est d\'efini r\'ecursivement par :
$$\Omega_k(P):= \Omega_{k-1}(P), \Lambda_{k-1}, T^{\omega_{k-1}}(P/T)^{r^{k-1}}, -\overline{\Lambda}_{k-1},  -\overline{\Omega}_{k-1}.$$

Si $z$ est la fraction continue infinie $z:=[P,T+1,T-1,z_4]$ satisfaisant :
\begin{equation}\label{genbsr} z^r=T^2z_4+(T+1),
\end{equation}
alors la suite de quotients partiels de $z$ est $\Omega_{\infty}(P)$, o\`u  $\Omega_{\infty}(P)$ est la suite infinie commen\c{c}ant par $\Omega_k(P)$, pour tout $k\geq 1$.
\end{theo}

\bigskip
\begin{remarque}  L'existence de la fraction continue $z$ v\'erifiant  (\ref{genbsr})  d\'ecoule aussi du th\'eor\`eme 1 de l'article \cite{ffa2008}. De plus, il est facile de voir que $z$ satisfait l'\'equation alg\`ebrique :
$$T^2z^{r+1}=(PT^2+T-1)z^r+1.$$
\end{remarque}

\section*{Remerciements} 
Je tiens \`a remercier tout particuli\`erement Alain Lasjaunias, qui m'a indiqu\'e ce sujet, pour les nombreuses discussions que nous avons eues et pour ses commentaires fort utiles au cours de mon travail. Je remercie \'egalement mon directeur de th\`ese, Boris Adamczewski, pour ses remarques judicieuses qui m'ont beaucoup aid\'e dans la r\'{e}daction du
pr\'esent travail.

\end{document}